\documentclass[11pt]{article}
\usepackage{amsmath,amsthm,amssymb,multicol}
\usepackage[colorlinks,bookmarksnumbered=false]{hyperref}
\usepackage{bookmark} 
\usepackage[all]{xy}
\usepackage{mathrsfs}

\title{On the period of Lehn, Lehn, Sorger, \\ and van Straten's symplectic eightfold}
\author{Nicolas Addington and Franco Giovenzana}
\date{}

\newcommand \A {\mathcal A}
\newcommand \C {\mathbb C}
\renewcommand \O {\mathcal O}
\renewcommand \P {\mathbb P}
\newcommand \Q {\mathbb Q}
\newcommand \Z {\mathbb Z}

\DeclareMathOperator \Br {Br}
\DeclareMathOperator \td {td}
\DeclareMathOperator \rank {rank}
\DeclareMathOperator \Gr {Gr}
\DeclareMathOperator \Pic {Pic}
\DeclareMathOperator \ch {ch}
\DeclareMathOperator \disc {disc}

\newcommand \Ktop {K_{\rm top}}
\newcommand \Knum {K_{\rm num}}
\newcommand \prim {{\rm prim}}
\newcommand \Xfour {X^{[4]}}

\newtheorem{thm}{Theorem}
\newtheorem{prop}[thm]{Proposition}
\newtheorem{lem}[thm]{Lemma}
\newtheorem{conj}[thm]{Conjecture}

\renewcommand \phi \varphi

\begin{document}

\maketitle

\begin{abstract}
\noindent For the irreducible holomorphic symplectic eightfold $Z$ associated to a cubic fourfold $Y$ not containing a plane, we show that a natural Abel--Jacobi map from $H^4_\prim(Y)$ to $H^2_\prim(Z)$ is a Hodge isometry.  We describe the full $H^2(Z)$ in terms of the Mukai lattice of the K3 category $\A$ of $Y$.  We give numerical conditions for $Z$ to be birational to a moduli space of sheaves on a K3 surface or to Hilb$^4$(K3).  We propose a conjecture on how to use $Z$ to produce equivalences from $\A$ to the derived category of a K3 surface.

\end{abstract}

\section*{Introduction}
Beauville and Donagi \cite{bd} showed that if $Y\subset \P^5_\C$ is a smooth cubic fourfold, then the variety $F$ of lines on $Y$ is an irreducible holomorphic symplectic fourfold, deformation-equivalent to the Hilbert scheme of two points on a K3 surface.  They showed moreover that the universal line $P \subset Y \times F$ induces a Hodge isometry
\[ [P]_*\colon H^4_\prim(Y,\Z) \xrightarrow{\,\sim\,} H^2_\prim(F,\Z), \]
with the intersection pairing on the left-hand side and the opposite of the Beauville--Bogomolov--Fujiki pairing on the right.

The first author later refined this result by describing the full $H^2(F,\Z)$ in terms of the Mukai lattice of the K3 category of $Y$, and characterized the cubics for which $F$ is birational to the Hilbert scheme of two points on a K3 surface, or more generally to a moduli space of sheaves on a K3 surface \cite{hassett_vs_galkin}.

In his survey paper \cite{hassett_survey}, 
Hassett remarked that it would be useful to have analogous results for the symplectic eightfold $Z$ constructed from twisted cubics on $Y$ by Lehn, Lehn, Sorger, and van Straten \cite{llsvs}.  This is the subject of the present paper.  We begin with a statement on primitive cohomology.

\begin{thm} \label{H2prim}
Let $Y$ be a smooth cubic fourfold not containing a plane.  Let
\[ u\colon M \to Z \]
be the contraction from the ten-dimensional space of generalized twisted cubics on $Y$ to the LLSvS symplectic eightfold.  Let $C \subset Y \times M$ be the universal curve.  Then the pullback
\[ u^*: H^2(Z,\Z) \to H^2(M,\Z) \]
is injective, and the map
\begin{equation} \label{map1}
[C]_*\colon H^4(Y,\Z) \to H^2(M,\Z)
\end{equation}
restricts to a Hodge isometry
\begin{equation} \label{iso1}
[C]_*\colon H^4_\prim(Y,\Z) \xrightarrow{\,\sim\,} u^*(H^2_\prim(Z,\Z)),
\end{equation}
with the intersection pairing on the left-hand side and the opposite of the Beauville--Bogomolov--Fujiki pairing on the right.
\end{thm}
\medskip

For a statement about the full $H^2(Z,\Z)$, we must recall Kuznetsov's K3 category
\[ \A = \langle \O_Y, \O_Y(h), \O_Y(2h) \rangle^\perp \subset D^b(Y), \]
the Mukai lattice introduced in \cite[\S2]{at},
\[ \Ktop(\A) = \{ [\O_Y], [\O_Y(h)], [\O_Y(2h)] \}^\perp \subset \Ktop(Y), \]
and the two special classes $\lambda_1, \lambda_2 \in \Ktop(\A)$.  For a recent account of the lattice theory of cubic fourfolds, see Huybrechts' survey paper \cite{huybrechts_survey}.

\begin{thm} \label{H2full}
Continue the notation of Theorem \ref{H2prim}.  Let
\[ \Phi^K\colon \Ktop(\A) \subset \Ktop(Y) \to \Ktop(M) \]
be the map on topological $K$-theory induced by the Fourier--Mukai kernel $I_C^\vee(-3h) \in D^b(Y \times M)$.  Then the map
\begin{equation} \label{map2}
c_1\circ \Phi^K :\Ktop(\A) \to H^2(M, \Z)
\end{equation}
restricts to a Hodge isometry
\begin{equation}
\langle\lambda_2-\lambda_1\rangle^{\perp} \xrightarrow{\,\sim\,} u^*(H^2(Z,\Z)), \label{iso2}
\end{equation}
with the Euler pairing on the left-hand side and the opposite of the Beauville--Bogomolov--Fujiki pairing on the right.
\end{thm}
\noindent One of the main ideas of \cite{al} is that $Z$ is a moduli space of complexes in $\A$ containing the projections of $\O_y$ for points $y \in Y$; the K-theory class of these complexes is $\lambda_2-\lambda_1$, so \eqref{iso2} is natural in light of O'Grady's description of the period of a moduli space of sheaves on a K3 surface \cite{og}.  Li, Pertusi, and Zhao developed this idea further using Bridgeland stability conditions in \cite{lpz}; note that their class $2\lambda_1+\lambda_2$ is related to our $\lambda_2-\lambda_1$ by an auto\-equivalence of $\A$.  Statements similar to Theorem \ref{H2full} appear in \cite[Prop.~5.2]{lpz} and \cite[Thm.~29.2(2)]{blamnups}.  We offer the present paper nonetheless because our proofs require less machinery, the statement in Theorem \ref{H2prim} is very classical and explicit, and we thought it worthwhile to articulate Theorem \ref{birational} and Conjecture \ref{the_conj} below.\footnote{After the first version of this paper was posted, Li, Pertusi, and Zhao informed us that the submitted version of \cite{lpz} includes a statement analogous to our Theorem 3(c).}\bigskip

Next we recall Hassett's Noether--Lefschetz divisors $\mathcal C_d$ in the moduli space $\mathcal C$ of cubic fourfolds, indexed by an integer $d$ satisfying
\begin{equation} \label{condition*} \tag{$*$}
d > 6 \text{ and } d \equiv 0 \text{ or } 2 \!\!\!\!\pmod 6.
\end{equation}
Again we refer to the survey papers of Huybrechts \cite{huybrechts_survey} and Hassett \cite{hassett_survey}.  From \cite{at}, \cite{blamnups}, and \cite{hassett_vs_galkin} we know that the K3 category $\A$ is equivalent to the derived category of coherent sheaves on a K3 surface, and the fourfold $F$ is birational to a moduli space of sheaves on a K3 surface, if and only if the cubic $Y \in \mathcal C_d$ for some $d$ satisfying
\begin{equation} \label{condition**} \tag{$**$}
\text{$d/2$ is not divisible by 9 or any prime $p \equiv 2 \!\!\!\!\pmod 3$,}
\end{equation}
or equivalently,
\begin{equation} \tag{$**$}
\text{$d$ divides $2n^2+2n+2$ for some $n \in \Z$}.
\end{equation}
%
From \cite{huybrechts_twisted} and \cite{blamnups} we know that the K3 category $\A$ is equivalent to the derived category of \emph{twisted} sheaves on a K3 surface, and $F$ is birational to a moduli space of twisted sheaves, if and only if $d$ satisfies the weaker condition
\begin{multline} \label{condition**'} \tag{$**'$}
\text{In the prime factorization of } d/2, \\[-1ex]
\text{primes } p \equiv 2 \!\!\!\!\pmod 3 \text{ appear with even exponents.}
\end{multline}
On the other hand, from \cite{hassett_vs_galkin} we know that $F$ is birational to the Hilbert scheme of two points on a K3 surface if and only if $d$ satisfies the stronger condition
\begin{equation} \label{condition***} \tag{$*{*}*$}
d\text{ is of the form }\frac{2n^2+2n+2}{a^2}\text{ for some }n,a \in \Z.
\end{equation}
Let us introduce a new condition
\begin{equation} \label{condition***'} \tag{$*{*}*'$}
d\text{ is of the form }\frac{6n^2+6n+2}{a^2}\text{ for some }n,a \in \Z.
\end{equation}
It is strictly stronger than \eqref{condition**}
but incomparable to \eqref{condition***}: the first few $d$s satisfying \eqref{condition***} are 14, 26, 38, 42, 62, and 86, whereas the first few $d$s satisfying \eqref{condition***'} are 14, 38, 62, 74, and 86.

\begin{thm} \label{birational}
Let $Y$ be a cubic fourfold not containing a plane --- that is, not lying in $\mathcal C_8$ --- so the symplectic eightfold $Z$ is defined.
\begin{enumerate}
\item $Z$ is birational to a moduli space of sheaves on a K3 surface if and only if $Y \in \mathcal C_d$ for some $d$ satisfying \eqref{condition**}.
\item $Z$ is birational to a moduli space of twisted sheaves on a K3 surface if and only if $Y \in \mathcal C_d$ for some $d$ satisfying \eqref{condition**'}.
\item $Z$ is birational to the Hilbert scheme of four points on a K3 surface if and only if $Y \in \mathcal C_d$ for some $d$ satisfying \eqref{condition***'}.
\end{enumerate}
\end{thm}\medskip

From \cite[Thm.~1.2(c)]{bm} we know that if $Z$ is birational to a moduli space of sheaves on a K3 surface, then $Z$ is \emph{isomorphic} to a moduli space of stable complexes on the same surface for some Bridgeland stability condition.  This suggests an alternative approach to proving that $\A$ is equivalent to the derived category of a K3 surface if and only if $d$ satisfies \eqref{condition**}:\footnote{On the other hand, if one assumes this fact (that $\A$ is geometric if and only if $d$ satisfies \eqref{condition**}), then Conjecture \ref{the_conj} seems to follow from \cite[Prop.~5.7]{lpz}.}
\begin{conj} \label{the_conj}
Suppose that $Y \in \mathcal C_d$ for some $d$ satisfying \eqref{condition**}, so $Z$ is isomorphic to a moduli space of Bridgeland-stable complexes on some K3 surface $S$.  Let $\alpha \in \Br(Z)$ be the obstruction to the existence of a universal complex on $Z \times S$, and let $U \in D^b(Z \times S, \alpha \boxtimes 1)$ be a twisted universal complex.  Recall that $Y$ is naturally embedded in $Z$, and consider the restriction $U|_{Y \times S}$, which can be untwisted because $\Br(Y) = 0$.  Then the functor $\A \to D^b(S)$ induced by $U|_{Y \times S}$ is an equivalence.
\end{conj}
\noindent We also expect a similar statement with \eqref{condition**'} and twisted K3 surfaces.  To prove Conjecture \ref{the_conj} would require a careful analysis of the ``wrong-way slices'' $U|_{Z \times x}$ for $x \in S$, and of their restrictions to $Y$.

\paragraph{Acknowledgments.}N.A.\ thanks the organizers of the Japanese--European Symposium on Symplectic Varieties at RIMS, Kyoto in October 2015, where these ideas first took shape.  We thank M.~Lehn, Ch.~Lehn, G.~Mongardi, J.~C.~Ottem, R.~Takahashi, L.~Giovenzana, and D.~C.~Veniani for valuable conversations, and A.~Kuznetsov for detailed comments on an earlier version of the paper.  N.A.\ was partially supported by NSF grant no.\ DMS-1902213.  F.G.\ was supported by DFG research grants Le 3093/2-1 and Le 3093/2-2.
\addtolength \textheight {2em} 

\section{Proof of Theorems \ref*{H2prim} and \ref*{H2full}}

\subsection{Recollections and reductions} \label{reductions}

For a smooth cubic fourfold $Y$ not containing a plane, Lehn, Lehn, Sorger, and van Straten considered the irreducible component $M \subset \operatorname{Hilb}^{3n+1}(Y)$ containing twisted cubics, which is a smooth tenfold \cite[Thm.~A]{llsvs}.  They produced a contraction $u\colon M \to Z$ onto an irreducible holomorphic symplectic eightfold, and a copy of $Y$ naturally embedded in $Z$ [ibid., Thm.~B].  Addington and Lehn showed that $Z$ is deformation-equivalent to the Hilbert scheme of four points on a K3 surface \cite{al}.

\begin{lem} \label{inj-lemma}
The pullback map $u^*\colon H^2(Z,\Z)\to H^2(M,\Z)$ is injective.
\end{lem}
\begin{proof}
The map $u$ factors as the blowup $Z' \to Z$ along $Y \subset Z$, and a $\P^2$-fibration $M \to Z'$, both of which induce injections on cohomology.  Alternatively, $u$ is surjective, so $u^*\colon H^*(Z,\Q) \to H^*(M,\Q)$ is injective by \cite[Lem.~7.28]{voisin_book}, and $H^2(Z,\Z)$ is torsion-free because $Z$ is simply-connected.
\end{proof}

The map \eqref{map1} is clearly a map of Hodge structures; let us argue that \eqref{map2} is as well.  For basics about the map on topological K-theory induced by a Fourier--Mukai kernel, and its compatibility with the more usual map on rational cohomology, we refer to \cite[\S2.1]{at}.  The Hodge structure on $\Ktop(\A) \subset \Ktop(Y)$ is pulled back via the Mukai vector $v\colon \Ktop(Y) \to \bigoplus H^{2i}(Y,\Q)$.  We have a commutative diagram
\begin{equation} \label{Phi_commutes_with_v}
\begin{gathered} 
\xymatrix{
\Ktop(\A) \ar[r]^{\Phi^K} \ar[d]_v & \Ktop(M) \ar[d]_v \\
H^*(Y,\Q) \ar[r]^{\Phi^H} & H^*(M,\Q),
}
\end{gathered}
\end{equation}
where $\Phi^K$ and $\Phi^H$ are the maps induced by $I_C^\vee(-3h) \in D^b(Y \times M)$.  Thus we see that $v \circ \Phi^K$ is a map of Hodge structures.  To get the map $c_1 \circ \Phi^K$ of \eqref{map2}, we first do $v \circ \Phi^K$, then multiply by $(\td M)^{-1/2}$ which preserves Hodge type, then take the degree-2 part.

Now it is a topological question whether the restrictions \eqref{iso1} and \eqref{iso2} take values in the subgroups claimed, and whether they are isometries, so it is enough to prove the claims for a single cubic $Y$.  In a bit more detail, let $U \subset \P\left(H^0(\O_{\mathbb{P}^5}(3))\right) = \P^{55}$ be the Zariski open set of cubic fourfolds that are smooth and do not contain a plane.  The construction of $M$ and $Z$ works in families, as discussed in the introduction to \cite{llsvs}, so we get smooth families $\mathcal{Y} \to U$, $\mathcal{M} \to \mathcal{Z} \to U$, and $\mathscr C\subset \mathcal{M} \times_U \mathcal{Y}$.  Then \eqref{iso1} and \eqref{iso2} are specializations of maps of local systems over $U$, so to prove the claims about them it is enough to do so at one point $Y \in U$.

In \S\S\ref{pfaff_sec}--\ref{sec1.3} we do this for \eqref{iso2} when $Y$ is a Pfaffian cubic fourfold whose associated K3 surface $X$ has Picard rank 1.  In \S\ref{2to1} we do it for \eqref{iso1} when $Y$ does not belong to any Noether--Lefschetz divisor.
\addtolength \textheight {-2em} 

\subsection{Proof of Theorem \ref*{H2full} for a very general Pfaffian cubic} \label{pfaff_sec}

To a generic choice of six two-forms on $\C^6$, Beauville and Donagi \cite{bd} associate a ``Pfaffian'' cubic fourfold $Y \subset \P^5$ and a K3 surface $X \subset \P^8$ that parametrizes a complete family of quartic scrolls on $Y$.  In \cite{al} the family is called $\Gamma \subset X \times Y$; it is generically 4-to-1 over $Y$.  We will recall more details in the next section, but we do not need them yet.

Assuming that $Y$ does not contain a plane, we continue to let $M$ be the space of generalized twisted cubics on $Y$, $u\colon M \to Z$ the contraction onto the LLSvS eightfold, and $C \subset Y \times M$ the universal curve.

We will finish proving Theorem \ref{H2full} under the assumption that $X$ has Picard rank 1, which by our discussion above is enough to prove the theorem in general.  We will construct a commutative diagram of Fourier--Mukai functors between $D^b(X)$ and $D^b(M)$, and deduce from it a commutative diagram of maps between lattices, most of which are isomorphisms, and ultimately compare the map \eqref{map2} with the composition
\[ [I_\xi]^\perp \subset \Ktop(X) \xrightarrow{I_\Xi^\vee} \Ktop(\Xfour) \xrightarrow{c_1} H^2(\Xfour, \Z), \]
which is known to be a Hodge isometry; here $\xi \subset X$ is a subscheme of length 4 and $\Xi \subset X \times \Xfour$ is the universal subscheme.   For standard facts about Fourier--Mukai kernels, their adjoints, induced maps on cohomology, etc.\ we refer to Huybrechts' book \cite[Ch.~5]{huybrechts_fm}. \bigskip

We consider the convolution
\[ T := I_\Gamma \circ I_C(2h) \in D^b(X \times M). \]
and recall some facts proved in \cite[\S3]{al}, with a few improvements: \smallskip
\begin{enumerate}
\item There is a non-empty, Zariski open set $M_0 \subset M$ such that the restriction $T[1]|_{X \times M_0}$ is quasi-isomorphic to an $M_0$-flat family of ideal sheaves of length-4 subschemes of $X$.  We take $M_0$ as big as possible. \smallskip
\item There is an open set $Z_0 \subset Z$ such that $M_0 = u^{-1}(Z_0)$. \smallskip

Proof: From \cite[Prop.~2]{al} we know that $M_0$ is a union of fibers of $u$, so $M_0 = u^{-1}(Z_0)$ for some $Z_0 \subset Z$, maybe not open \emph{a priori}.  But $u$ is surjective, so $Z \setminus Z_0 = u(M \setminus M_0)$, and $u$ is proper, hence closed, so $Z \setminus Z_0$ is closed.\smallskip
\item The classifying map $t'\colon M_0 \to \Xfour$ descends to an open immersion $t\colon Z_0 \to \Xfour$. \smallskip

Proof: From \cite[\S3]{al} we know that $t'$ descends to a map $t$ that is injective.  But an injective holomorphic map between complex manifolds of the same dimension is an open immersion by the proof of \cite[Prop. on p.~19]{gh}; see also \cite{rosay}.  (So the smaller open set $Z_1 \subset Z_0$ of \cite[\S3]{al} is unnecessary.) \smallskip
\end{enumerate}
The key to the present proof will be the following fact, whose proof we postpone to the next section:
\begin{prop} \label{key_prop}
If the K3 surface $X$ has Picard rank 1, then the open set $Z_0$ contains $Y$, and its complement $Z \setminus Z_0$ has codimension at least 2 in $Z$.
\end{prop}
\noindent Thus $M \setminus M_0$ has codimension at least 2 in $M$ as well, because $u$ is a $\P^2$-fibration over $Z \setminus Y$; and thus the inclusions $Z_0 \hookrightarrow Z$ and $M_0 \hookrightarrow M$ induce isomorphisms on $H^2$.

The reader may wonder why we need Proposition \ref{key_prop}, when we can say immediately that because $Z$ and $\Xfour$ have nef canonical bundles, the birational map $Z \dashrightarrow \Xfour$ is biregular on an open set $W$ containing $Z_0$ whose complement has codimension at least 2 \cite[Cor.~3.54]{km}.  The reason is that we would not know anything about the restriction of $T$ to $X \times u^{-1}(W)$, which we need in what follows.
\bigskip

Consider the functor $D^b(M) \to D^b(Y)$ induced by $I_C \otimes \O_Y(2h)$, the functor $D^b(Y) \to D^b(X)$ induced by $I_\Gamma$, their composition
\begin{equation} \label{first_comp}
D^b(M) \xrightarrow{I_C(2h)} D^b(Y) \xrightarrow{I_\Gamma} D^b(X),
\end{equation}
and its left adjoint 
\begin{equation} \label{left_adj}
D^b(X) \xrightarrow{I_\Gamma^\vee[2]} D^b(Y) \xrightarrow{I_C^\vee(-5h)[4]} D^b(M),
\end{equation}
which it will later be convenient to rewrite as
\begin{equation} \label{left_adj'} \tag{$7'$}
D^b(X) \xrightarrow{I_\Gamma^\vee(-2h)[2]} D^b(Y) \xrightarrow{I_C^\vee(-3h)[4]} D^b(M).
\end{equation}
We post-compose with the restriction
\begin{equation} \label{with_restriction}
D^b(X) \xrightarrow{I_\Gamma^\vee(-2h)[2]} D^b(Y) \xrightarrow{I_C^\vee(-3h)[4]} D^b(M) \xrightarrow{i^*} D^b(M_0).
\end{equation}
The first composition \eqref{first_comp} is induced by $T$, so the left adjoint \eqref{left_adj} or \eqref{left_adj'} is induced by $T^\vee[2]$, so the restriction \eqref{with_restriction} is induced by $T^\vee|_{X \times M_0}[2]$.

Because $T[1]|_{X \times M_0}$ is an $M_0$-flat family of ideal sheaves with classifying map $t'\colon M_0 \to \Xfour$, there is a line bundle $\mathcal L$ on $M_0$ such that
\[ T[1]|_{X \times M_0} = (1 \times t')^* I_\Xi \otimes \mathcal L, \]
where again $\Xi \subset X \times \Xfour$ is the universal subscheme.\footnote{For a careful proof that Hilbert schemes are moduli spaces of stable sheaves when $H^1(\O) = 0$, see \cite[Lem.~B.5.6]{kps}.}  So the composition \eqref{with_restriction} agrees with
\[ D^b(X) \xrightarrow{I_\Xi^\vee[3]} D^b(\Xfour) \xrightarrow{t'^*} D^b(M_0) \xrightarrow{\otimes \mathcal L^\vee} D^b(M_0). \]
Because $t' = t \circ u$, this is
\[ D^b(X) \xrightarrow{I_\Xi^\vee[3]} D^b(\Xfour) \xrightarrow{t^*} D^b(Z_0) \xrightarrow{u^*} D^b(M_0) \xrightarrow{\otimes \mathcal L^\vee} D^b(M_0). \]
Passing to topological K-theory \cite[\S2.1]{at}, and taking first Chern classes where shown, we get a diagram
\begin{equation} \label{five_steps}
\xymatrix{
\Ktop(X) \ar[r]^-{-I_\Xi^\vee} &
\Ktop(\Xfour) \ar[r]^-{t^*} \ar[d]_{c_1} &
\Ktop(Z_0) \ar[r]^-{u^*} \ar[d]_{c_1} &
\Ktop(M_0) \ar[r]^-{\cdot[\mathcal L^\vee]} \ar[d]_{c_1} &
\Ktop(M_0) \ar[d]_{c_1} \\
&
H^2(\Xfour,\Z) \ar[r]^-{t^*}_-{\sim} &
H^2(Z_0,\Z) \ar@{^(->}[r]^-{u^*} &
H^2(M_0,\Z) \ar@{-->}[r] &
H^2(M_0,\Z),
} \end{equation}
where there is no well-defined map to fill in the dashed arrow, and the injectivity of the lower $u^*$ follows from Lemma \ref{inj-lemma} and Proposition \ref{key_prop}. 

The map $D^b(\Xfour) \xrightarrow{I_\Xi} D^b(X)$ takes the the skyscraper sheaf of a point to $I_\xi$, where $\xi \subset X$ is a subscheme of length 4, so the left adjoint $I_\Xi^\vee[2]$ and its induced map on K-theory
\[ \Ktop(X) \xrightarrow{I_\Xi^\vee} \Ktop(\Xfour) \]
takes classes in $[I_\xi]^\perp$ to classes of rank 0.  Thus in the big diagram \eqref{five_steps}, if we replace $\Ktop(X)$ with the subspace $[I_\xi]^\perp$ then multiplying by $[\mathcal L^\vee]$ does not affect the first Chern class, so we can forget the last column of the diagram.

The composition
\[ [I_\xi]^\perp \subset \Ktop(X) \xrightarrow{I_\Xi^\vee} \Ktop(\Xfour) \xrightarrow{c_1} H^2(\Xfour, \Z) \]
is a Hodge isometry by O'Grady's classic calculation \cite{og}.  Moreover the map $t^*$ is an isomorphism on $H^2$, so the big diagram \eqref{five_steps} gives a Hodge isometry from $[I_\xi]^\perp \subset \Ktop(X)$ to $u^*(H^2(Z_0,\Z)) \subset H^2(M_0,\Z)$.  Because $M \setminus M_0$ and $Z \setminus Z_0$ have codimension at least 2, this is the same as $u^*(H^2(Z,\Z)) \subset H^2(M,\Z)$.

If we take \eqref{with_restriction}, pass to topological K-theory, and include first Chern classes, we get a diagram
\begin{equation} \label{second_big_diagram}
\begin{gathered} \xymatrix{
\Ktop(X) \ar[r]^-{I_\Gamma^\vee(-2h)} &
\Ktop(Y) \ar[r]^-{I_C^\vee(-3h)} &
\Ktop(M) \ar[r]^-{i^*} \ar[d]_{c_1} &
\Ktop(M_0) \ar[d]_{c_1} \\
& &
H^2(M,\Z) \ar[r]^-{i^*}_-{\sim} &
H^2(M_0,\Z).
} \end{gathered}
\end{equation}
In \cite[Prop.~3]{al} it is proved that the functor
\[ D^b(X) \xrightarrow{I_\Gamma^\vee(-3h)[4]} D^b(Y) \]
is fully faithful, and its image is $\langle \O_Y(-h), \O_Y, \O_Y(h) \rangle^\perp$.  In this paper we are taking $\A = \langle \O_Y, \O_Y(h), \O_Y(2h) \rangle^\perp$, so we should use the equivalence
\[ D^b(X) \xrightarrow{I_\Gamma^\vee(-2h)[4]} \A. \]
We claim that the induced map on topological K-theory takes $[I_\xi]$ to our class $\lambda_2 - \lambda_1$.  To see this, first observe that the functor $D^b(Y) \xrightarrow{I_\Gamma(-h)} D^b(X)$ takes the skyscraper sheaf of a point $\O_y$ to some $I_\xi$, so its right adjoint $D^b(X) \xrightarrow{I_\Gamma^\vee(-2h)[4]} D^b(Y)$ takes $I_\xi$ to the projection of $\O_y$ into $\A$.  In K-theory we have $[\O_y] = [\O_\ell(2)] - [\O_\ell(1)]$, where $\ell$ is a line on $Y$, so the class of the projection of $\O_y$ is $\lambda_2 - \lambda_1$.

Thus in our second big diagram \eqref{second_big_diagram}, the first map takes $[I_\xi]^\perp$ isometrically to $(\lambda_2 - \lambda_1)^\perp \subset \Ktop(\A) \subset \Ktop(Y)$.  We conclude that the map \eqref{map2} takes $(\lambda_2 - \lambda_1)^\perp$ isometrically to $u^*(H^2(Z,\Z))$, as desired.

\subsection{Proof of Proposition \ref*{key_prop}} \label{sec1.3}

Now we recall the construction of the Pfaffian cubic fourfold $Y$, the K3 surface $X$, and the correspondence $\Gamma \subset X \times Y$ in detail.  Fix a vector space $V \cong \C^6$ and a generic 6-dimensional subspace $L \subset \Lambda^2 V^*$.  The cubic is
\[ Y = \Bigl\{ [\varphi] \in \P(L) \Bigm| \rank(\phi) = 4 \Bigr\} = \Bigl\{ [\varphi] \in \P(L) \Bigm| \varphi^3 = 0 \Bigr\}, \]
the K3 surface is
\[  X=\Bigl\{[P]\in\Gr(2,V) \Bigm| \varphi|_P=0\text{ for all }\phi \in L \Bigr\}, \]
and the correspondence is
\[ \Gamma = \Bigl\{ ([P], [\varphi]) \in X\times Y \Bigm| P\cap \ker\varphi \ne 0 \Bigr\}. \]
For a generic choice of $L$, both $X$ and $Y$ are smooth, $X$ does not contain a line, and $Y$ does not contain a plane.

In \cite[Lem.~4]{al} it is proved that $\Gamma$ is generically 4-to-1 over $Y$.  We improve this as follows:
\begin{lem} \label{Gamma_flat}
If $X$ does not contain a $(-2)$-curve then $\Gamma$ is flat over $Y$.
\end{lem}
\begin{proof}
Let $\Gamma_\varphi \subset X$ be the scheme-theoretic fiber of $\Gamma$ over a point $[\varphi] \in Y$.  We will argue that $\Gamma_\varphi$ is zero-dimensional; then by the proof of \cite[Lem.~4]{al} it has length 4.

Consider the Schubert cycle
\[ \Sigma_\varphi = \Bigl\{ [P] \in \Gr(2,V) \Bigm| P \cap \ker \varphi \ne 0 \Bigr\} \]
of which $\Gamma_\varphi$ is a linear section, and the normalization
\[ \tilde\Sigma_\varphi = \Bigl\{ (l,P) \in \P(\ker\varphi) \times \Gr(2,V) \Bigm| l \subset P \Bigr\}. \]
We observe that $\tilde\Sigma_\varphi$ is a $\P^4$-bundle over $\P(\ker\varphi) = \P^1$.  First we claim that the preimage of $\Gamma_\varphi$ in $\tilde\Sigma_\varphi$ meets each $\P^4$ fiber in at most one point, even scheme-theoretically: to see this, note that the $\P^4$s are embedded linearly in $\Gr(2,V) \subset \P(\Lambda^2 V)$, so if a linear section contained more than one point of a $\P^4$ it would contain a line, contradicting our hypothesis on $X$.

Next we claim that the preimage of $\Gamma_\varphi$ in $\tilde\Sigma_\varphi$ meets at most finitely many $\P^4$ fibers.  Otherwise it would meet them all, hence would give a section of this $\P^4$-bundle over $\P^1$, hence a smooth rational curve on $X$, again contradicting our hypothesis.

Thus the preimage of $\Gamma_\varphi$ in $\tilde\Sigma_\varphi$ is zero-dimensional, so $\Gamma_\varphi$ is as well.
\end{proof}

\begin{lem} \label{Pullbacks}
Suppose that $X$ has Picard rank 1, so in particular $X$ contains no $(-2)$-curve and thus $\Gamma$ induces a regular map $j\colon Y \to \Xfour$.  Let $\tilde H$ and $B$ be the basis for $\Pic(\Xfour)$ used in \cite[\S13]{bm}.  Then
\begin{align*}
j^* B &= 9 h \\
j^* \tilde H &= 14 h
\end{align*}
where $h \in \Pic(Y)$ is the hyperplane class.
\end{lem}
\begin{proof}
By construction of $j\colon Y \to \Xfour$ we have a Cartesian diagram
\[\xymatrix{
\Gamma \ar[r] \ar[d]^{p} & \Xi \ar[d]^{p'} \\
Y \ar[r]^j & \Xfour .
} \]
By \cite[Lem.~3.7]{lehn} we have $B = -c_1 (p'_* \O_\Xi)$, and thus
\[ j^*B = -c_1(p_* \O_\Gamma). \]
Applying the Grothendieck--Riemann--Roch formula to the embedding \linebreak $i\colon Y \hookrightarrow \P^5$, we get
\begin{equation} \label{to_compute}
i_* \ch(p_* \O_\Gamma) = \ch(i_* p_* \O_\Gamma)\cdot \td(\O_{\P^5}(3)).
\end{equation}
Let $0 \to \mathcal P \to \O_X \otimes V \to \mathcal Q \to 0$ be the restriction to $X$ of the  tautological bundle sequence on $\Gr(2,V)$.  In \cite[\S2]{al} it is argued that $(1 \times i)(\Gamma) \subset X \times \P^5$ is the degeneracy locus of a map $\mathcal P \boxtimes \O_{\P^5} \to \mathcal Q^\vee \boxtimes \O_{\P^5}(1)$, of the expected dimension.  Thus with the Eagon--Northcott resolution of $(1 \times i)_* \O_\Gamma$ and a lot of Schubert calculus, we can compute the right-hand side of \eqref{to_compute}.  We carried out this computation by hand and checked it with the Schubert2 package of Macaulay2 \cite{M2}; we include our code in the ancillary file \verb|Gamma.m2|.  The result is
\[ i_* \ch(p_* \O_\Gamma) =
12h - 27h^2 + \frac{65}{2}h^3 - \frac{33}{2}h^4 + \frac{19}{8}h^5 \in H^*(\P^5,\Q). \]
Because $i_*(h) = 3h^2$, we read off $c_1(p_*\O_\Gamma) = -9h$, so
\[j^*B = 9h \in H^2(Y,\Q). \]

To find $j^* \tilde H$, let $H \subset X$ be a hyperplane section, and let $q'\colon \Xi \to X$ be the natural projection.  Then $p'_* q'^* \O_H$ is a sheaf supported on the locus of 4-tuples of which one is contained in $H$, that is, on $\tilde H$; and this sheaf has generic rank 1 on its support, so its first Chern class is $\tilde H$.  Using the resolution
\[ 0 \to \O_X(-H) \to \O_X \to \O_H \to 0 \]
we find that
\[ i_*\ch(p_*(\O_\Gamma \otimes q^*\O_H)) = 42h^2 - 91 h^3 + 56h^4 - \frac{35}{4}h^5. \]
Again the code is included in \verb|Gamma.m2|.  Taking the coefficient of $h^2$ and dividing by 3, we get $j^*(\tilde H) = 14h$ as desired. 
\end{proof}

\begin{lem} \label{eff_misses_Y}
Suppose that $X$ has Picard rank 1, so again $\Gamma$ induces a regular map $j\colon Y \to \Xfour$.  Then every effective divisor on $\Xfour$ meets $j(Y)$.
\end{lem}
\begin{proof}
First we calculate the movable cone of $\Xfour$, applying Bayer and Macr\`i's \cite[Prop.~13.1]{bm} with $d=7$ and $n=4$.  Then $d(n-1) = 21$ is not a perfect square; and $3X^2-7Y^2 = 1$ has no solution, as we see by reducing mod 3; so we are in case (c).  The minimal positive solution to $X^2 - 21Y^2 = 1$ is $X = 55$, $Y=12$.  This satisfies $X \equiv 1 \pmod 3$ rather than $X \equiv -1 \pmod 3$, contradicting \cite[eq.\,(33)]{bm}, but the latter contains a typo which Debarre corrects in \cite[Example~3.20]{debarre}.  Thus the movable cone is spanned by
\[ \tilde H \qquad \text{and} \qquad 55\tilde H - 84B. \]
The same information can in principle be obtained from work of Markman \cite{survey_of_torelli}, which does not use Bridgeland stability conditions, but \cite{bm} is more user-friendly.

Next we need to know that if $D$ is an effective divisor on $\Xfour$ and $M$ is movable, then $q(D,M) \ge 0$, where $q$ is the Beauville--Bogomolov--Fujiki pairing.  This follows from the proof of \cite[Thm.~7]{ht_movable}.  In fact a result of Boucksom \cite[Prop.~4.4]{boucksom} implies that the pseudo-effective cone is dual to the closure of the movable cone, but we do not need the full strength of this.

Now if an effective divisor $D$ on $\Xfour$ does not meet $j(Y)$, then $j^*(D) = 0$, so from Lemma \ref{Pullbacks} we know that $D$ is a multiple of $9\tilde H - 14B$.  But we have
\begin{align*}
q(\tilde H, \tilde H) &= 14
& q(\tilde H, B) &= 0
& q(B,B) &= -6,
\end{align*}
and thus
\begin{align*}
q(9\tilde H - 14B, \tilde H) &= 126 &
q(9\tilde H - 14B, 55\tilde H - 84B) = -126,
\end{align*}
so no multiple of $9\tilde H - 14B$ pairs non-negatively with both walls of the movable cone.
\end{proof}

We conclude this section by deducing Proposition \ref{key_prop} from the lemmas above.  From \cite[\S3]{al} we know that the open immersion $t\colon Z_0 \to \Xfour$, restricted to $Y_0 := Y \cap Z_0$, agrees with the map $j\colon Y_0 \to \Xfour$ induced by $I_\Gamma$.  If $X$ has Picard rank 1 then $Y_0 = Y$ by Lemma \ref{Gamma_flat}.  The open immersion $t\colon Z_0 \to \Xfour$ gives a birational map $Z \dashrightarrow \Xfour$.  Because $Z$ and $\Xfour$ are minimal models, this birational map is biregular on an open set $W$ containing $Z_0$ whose complement has codimension at least 2 in both $Z$ and $\Xfour$.  If $Z \setminus Z_0$ has any component of codimension 1, then so too does $W \setminus Z_0$, and taking its closure in $\Xfour$ we get an effective divisor that does not meet $j(Y)$, contradicting Lemma \ref{eff_misses_Y}.

\subsection{Proof of Theorem \ref*{H2prim} for a very general cubic} \label{2to1}

By our discussion in \S\ref{reductions}, to prove Theorem \ref{H2prim} it is enough to show that for a single cubic $Y$, the map \eqref{map1} takes $H^4_\prim(Y,\Z)$ into $u^*(H^2_\prim(Z,\Z))$ and respects the pairings.  We deduce this from Theorem \ref{H2full} when $Y$ is not in any Noether--Lefschetz divisor.

Because $Y$ is Noether--Lefschetz general, the transcendental lattice of $\langle \lambda_2-\lambda_1 \rangle^\perp \subset \Ktop(\A)$ is $\langle \lambda_1,\lambda_2 \rangle^\perp$.  The transcendental lattice of $H^2(Z,\Z)$ is contained in $H^2_\prim(Z,\Z)$, so the Hodge isometry \eqref{iso2} takes $\langle \lambda_1,\lambda_2 \rangle^\perp$ into $u^*(H^2_\prim(Z,\Z))$; but they are primitive sublattices of the same rank, so \eqref{iso2} gives an isomorphism between them.

So we need to compare the action of \eqref{map2} on $\langle \lambda_1,\lambda_2 \rangle^\perp$ with the action of $\eqref{map1}$ on $H^4_\prim(Y,\Z)$.  In the diagram \eqref{Phi_commutes_with_v}, we know from \cite[Prop.~2.3]{at} that $v$ takes $\langle \lambda_1,\lambda_2 \rangle^\perp \subset \Ktop(\A)$ isometrically to $H^4_\prim(Y,\Z) \subset H^*(Y,\Q)$.  To finish proving Theorem \ref{H2prim}, it remains to show that the maps
\[ H^4_\prim(Y,\Z) \to H^2(M,\Q) \]
induced by $v(I_C^\vee(-3h)) \in H^*(Y \times M, \Q)$ and $[C] \in H^6(Y \times M, \Z)$ are the same.  To see this, observe that
\begin{align*}
\ch(I_C) &= 1 - [C] + \text{higher-order terms}, \\
\ch(I_C^\vee) &= 1 + [C] + \text{higher-order terms}.
\end{align*}
What happens when we map a class $\alpha \in H^4_\prim(Y)$ to $H^*(M,\Q)$ using $v(I_C^\vee(-3h))$?  First we multiply by $(\td Y)^{1/2}$ and $\ch(\O_Y(-3h))$, which are polynomials in the hyperplane class $h \in H^2(Y,\Z)$ and thus have no effect on $\alpha$.  Then we pull up to $H^*(Y \times M, \Q)$, multiply with $\ch(I_C^\vee)$, and push down to $H^*(M,\Q)$, which yields $[C]_* \alpha + {}$higher-order terms.  Then we multiply by $(\td M)^{1/2}$, which does not affect the leading term $[C]_* \alpha$.  Then we take the degree-2 part, leaving $[C]_* \alpha$ as desired.

\section{Proof of Theorem \ref*{birational}}

We adapt an argument developed in \cite{hassett_vs_galkin} and \cite[Prop.~4.1]{huybrechts_twisted} and polished in \cite[\S3.2]{huybrechts_survey}, staying especially close to the latter reference.

From Verbitsky's hyperk\"ahler Torelli theorem \cite{verbitsky, huybrechts_torelli, looijenga_torelli} and Markman's work of the monodromy of manifolds of K3$^{[n]}$-type \cite[\S9]{survey_of_torelli}, we know that if $n-1$ is a prime power, then two manifolds $M$ and $M'$ of K3$^{[n]}$-type are birational (or bimeromorphic) if and only if there is a Hodge isometry $H^2(M,\Z) \cong H^2(M',\Z)$.  If $n-1$ has two or more prime factors there is a more subtle statement, but this does not concern us since the eightfold $Z$ is of K3$^{[4]}$-type.  \bigskip

By \cite[Prop.~1.24 and Prop.~1.13]{huybrechts_survey}, a cubic fourfold $Y$ is in $\mathcal C_d$ for some $d$ satisfying \eqref{condition**} if and only if the Mukai lattice $\Ktop(\A)$ is Hodge-isometric to the Mukai lattice $\tilde H(S,\Z)$ of a K3 surface, and similarly with \eqref{condition**'} and the Mukai lattice $\tilde H(S,\alpha,\Z)$ of a twisted K3 surface.

If $\varphi\colon \Ktop(\A) \to \tilde H(S,\Z)$ is a Hodge isometry, let $v = \varphi(\lambda_2 - \lambda_1)$, which is a primitive vector satisfying $v^2 = 6$.  Then for a $v$-generic polarization $h \in \Pic(S)$, the moduli space $M := M_h(v)$ of $h$-stable sheaves with Mukai vector $v$ is a smooth variety of K3$^{[4]}$-type, and satisfies $H^2(M,\Z) \cong v^\perp \subset \tilde H(S,\Z)$.  Thus
\[ H^2(Z,\Z) \cong \langle \lambda_2 - \lambda_1 \rangle^\perp \cong v^\perp \cong H^2(M,\Z). \]
Conversely, any Hodge isometry $\langle \lambda_2 - \lambda_1 \rangle^\perp \cong v^\perp$ extends to a Hodge isometry $\Ktop(\A) \to \tilde H(S,\Z)$ taking $\lambda_2-\lambda_1$ to $\pm v$ by \cite[Cor.~1.5.2]{nikulin}: the discriminant group of the lattice $v^\perp$ is $\Z/6$ whose automorphism group is $\pm 1$.  With twisted K3 surfaces the argument is similar.  This proves parts (i) and (ii) of Theorem \ref{birational}. \bigskip

To prove part (iii), by \cite[Prop.~5]{hassett_vs_galkin} or by the proof of \cite[Prop.~3.4(i)]{huybrechts_survey} we see that $Z$ is birational to $S^{[4]}$ for some K3 surface $S$ if and only if the algebraic lattice $\Knum(\A) \subset \Ktop(\A)$ contains a copy of the hyperbolic plane $U$ which contains $\lambda_2 - \lambda_1$; or equivalently, if and only if there there is a vector $w \in \Knum(\A)$ with $\chi(w,w) = 0$ and $\chi(w,\lambda_2-\lambda_1) = 1$.\footnote{Huybrechts \cite{huybrechts_survey} and many other authors define the Mukai pairing on $\Ktop(\A)$ as having the opposite sign from the Euler pairing, in order to agree with the Beauville--Bogomolov pairing on $H^2$ of various hyperk\"ahler varieties, whereas Addington and Thomas \cite{at,hassett_vs_galkin} define it having the same sign as the Euler pairing.  In effect, the first convention regards $\Ktop(\A)$ as a Hodge structure of weight 2, the second as a Hodge structure of weight 0.  In the calculation that follows, to avoid confusion, we only use Euler pairing and do not mention the Mukai pairing.}

If there is such a $w$, let $L = \langle \lambda_1, \lambda_2, w \rangle \subset \Knum(\A)$, and let $n = \chi(w, \lambda_1)$, so the Gram matrix of the Euler pairing on $L$ is
\[ \begin{pmatrix}
-2 & 1 & n \\
1 & -2 & n+1 \\
n & n+1 & 0
\end{pmatrix} \]
Thus $\disc(L) = 6n^2 + 6n + 2$.  Let $M$ be the saturation of $L$ in $\Knum(\A)$, let $a$ be the index of $L$ in $M$, and let $d = \disc(M)$; then $a^2 d = \disc(L)$, so $d = (6n^2+6n+2)/a^2$.  By \cite[Prop.~2.5]{at} we have $Y \in \mathcal C_d$.

Conversely, suppose that $Y \in \mathcal C_d$ for some $d$ of the form $(6n^2+6n+2)/a^2$.  We claim that $a \equiv 1 \pmod 6$: to see this, observe that $6n^2+6n+2$ is the norm of the primitive vector $(n,-n-1)$ in the $A_2$ lattice, hence satisfies \eqref{condition**} by \cite[Prop.~1.13(iii)]{huybrechts_survey}, so $a$ must be a product of primes $p \equiv 1 \pmod 3$.  Thus $d \equiv 2 \pmod 6$; write $d = 6k+2$ and $a = 3m+1$.  By \cite[Lem.~9]{hassett_vs_galkin} there is an element $\tau \in \Knum(\A)$ such that $\langle \lambda_1, \lambda_2, \tau \rangle$ is a primitive sublattice on which the Gram matrix of the Euler pairing is
\[ \begin{pmatrix}
-2 & 1 & 0 \\
1 & -2 & 1 \\
0 & 1 & 2k
\end{pmatrix}. \]
Then the vector
\[ w := (m-n)\lambda_1 + (2m-n)\lambda_2 + a\tau \]
satisfies $\chi(w,w) = 0$ and $\chi(w, \lambda_2-\lambda_1) = 1$, as desired.

\bibliographystyle{plain}
\bibliography{Buona}

\begin{multicols}{2}
\scriptsize
\noindent
Nicolas Addington \\
Department of Mathematics \\
University of Oregon \\
Eugene, OR 97403 \\
United States \\
adding@uoregon.edu
\columnbreak

\noindent Franco Giovenzana \\
Fakult\"at f\"ur Mathematik\\
Technische Universit\"at Chemnitz \\
Reichenhainer Strasse 39 \\
09126 Chemnitz \\
Germany \\
franco.giovenzana@mathematik.tu-chemnitz.de
\end{multicols}

\end{document}